\documentclass{amsart}
\usepackage{url} 
\usepackage{amsmath,amssymb,amsthm}
\usepackage{verbatim}
\usepackage{enumerate}
\usepackage{txfonts}
\numberwithin{equation}{section} 

\DeclareMathOperator{\gal}{Gal}
\DeclareMathOperator{\ord}{ord}

\DeclareMathOperator{\aut}{Aut}

\DeclareMathOperator{\rad}{rad}

\newtheorem{thm}{Theorem}[section]
\newtheorem{lem}[thm]{Lemma}
\newtheorem{cor}[thm]{Corollary}
\newtheorem{defn}[thm]{Definition}
\newtheorem{con}[thm]{Conjecture}
\newtheorem{prop}[thm]{Proposition}
\newtheorem{exa}[thm]{Example}
\newtheorem{rem}[thm]{Remark}

\begin{document}
\title{Perfect powers in elliptic divisibility sequences}
\author{Jonathan Reynolds}
\address{Mathematisch Instituut \\ Universiteit Utrecht \\ Postbus 80.010 \\ 3508 TA Utrecht \\ Nederland}
\email{J.M.Reynolds@uu.nl}
\date{\today}
\thanks{The author is supported by a Marie Curie Intra European Fellowship (PIEF-GA-2009-235210)}
\thanks{The author thanks Gunther Cornelissen, Sander Dahmen and Shaun Stevens for helpful comments}

\begin{abstract}
It is shown that there are finitely many perfect powers in an elliptic divisibility sequence whose first term is divisible by $2$ or $3$. For Mordell curves the same conclusion is shown to hold if the first term is greater than $1$. Examples of Mordell curves and families of congruent number curves are given with corresponding elliptic divisibility sequences having no perfect power terms. The proofs combine primitive divisor results with modular methods for Diophantine equations.                                        
\end{abstract} 

\maketitle

\section{Introduction}

Using modular techniques inspired by the proof of Fermat's Last Theorem, it was finally shown in \cite{MR2215137} that the only perfect powers in the Fibonnaci sequence are $1$, $8$ and $144$. Fibonnaci is just one example of an infinte sequence $(h_m)$ of integers  
\[
\ldots ,h_{-2}, h_{-1}, h_0, h_1, h_2, \ldots
\]
satisfying $h_m|h_n$ whenever $m|n$ and, up to sign,  
\[
h_{m+n}h_{m-n}=h_{m+1}h_{m-1}h_n^2-h_{n+1}h_{n-1}h_m^2
\]
for all $m,n \in \mathbb{Z}$, where $h_2h_3 \ne 0$. Gezer and Bizim
 \cite{MR2669714} have recently described the squares in some of these sequences but $(h_m)$ was first studied in general by Ward \cite{MR0023275} and is related to a Weierstrass equation 
\begin{equation} \label{we}
 y^2+a_1xy+a_3y=x^3+a_2x^2+a_4x+a_6
\end{equation}
with integer coefficients. See \cite{MR2514094, MR1171452} for background on Weierstrass equations and elliptic curves. The non-singular rational points on (\ref{we}) form a group $E_{ns}(\mathbb{Q})$ and for $P \in E_{ns}(\mathbb{Q})$ different from the identity we can write 
\begin{equation} \label{B_P}
(x(P),y(P))=\left(\frac{A_P}{B_P^2}, \frac{C_P}{B_P^3} \right),
\end{equation}
where $A_P,B_P,C_P \in \mathbb{Z}$ and $\gcd(A_PC_P,B_P)=1$. Let $(h_m)$ be a sequence of integers as above with $h_0=0$ and $h_1=1$. Building on work of Ward, Shipsey \cite{Shi01} has given a formula for a Weierstrass equation (\ref{we}) such that $h_m=\psi_m(0,0)$, where $\psi_m$ is the $m$th division polynomial (see Section \ref{FH}) and $h_m=\pm B_{m(0,0)}$ if $\gcd(a_3,a_4)=1$. For example, up to sign, $(0,0)$ on
\[
y^2+xy+y=x^3-2x^2
\] 
generates the Fibonacci sequence $(B_{m(0,0)})$. In \cite{MR0023275} Ward calls $(h_m)$ an elliptic divisibility sequence; however, as in the Fibonacci case, the Weierstrass equation for $(h_m)$ may have a singular point and so not define an elliptic curve. Via the work of Everest \cite{MR2164113, MR2045409}, Ingram \cite{MR2301226}, Silverman 
\cite{IngrSilv, MR961918} et al, it has now become conventional to use the following

\begin{defn} \label{EDS} Let $E/\mathbb{Q}$ be an elliptic curve and (\ref{we}) a Weierstrass equation for $E$.  Let $P \in E(\mathbb{Q})$ be a non-torsion point. For $m \in \mathbb{N}$ denote $B_{mP}$, as in (\ref{B_P}), by $B_m$. The sequence $(B_m)$ is an 
\emph{elliptic divisibility sequence}.
\end{defn}

In the current paper, we are interested in analogues for elliptic divisibility sequences (in the sense of Definition \ref{EDS}) to the result for Fibonacci numbers. There are certainly perfect powers in some elliptic divisibility sequences. For example,
\[
E: y^2+xy=x^3+x^2-7x+5
\]
with $P=(2,-3)$ gives $B_m=1$ for $m=1,2,3,4,7$ and $B_{12}=2^7$. However, the following theorem shows that one can often prove that there are only finitely many perfect powers in such sequences.

\begin{thm} \label{2or3} Let $(B_m)$ be an elliptic divisibility sequence  
whose first term is divisible by $2$ or $3$. There are finitely many perfect powers in $(B_m)$. Moreover, if $B_m=z^l$ for some integer $z$ and prime $l$ then $l$ can be bounded explicitly in terms of $E$ and $P$.   
\end{thm}

The proof of Theorem \ref{2or3} combines a recent Frey-Hellegouarch construction for Klein forms by Bennett and Dahmen \cite{BenDah} with a primitive divisor result due to Silverman \cite{MR961918}. The method of proof is so flexible that it also allows one in certain concrete cases to completely determine the set of all perfect power terms, as was done for the Fibonacci sequence (see Proposition \ref{D=11} and Example \ref{exa8} below). The condition that only $2$ or $3$ divides the first term is because higher primes, such as $5$, do not give a Klein form as in Defintion \ref{K}.      
 
Let $E/\mathbb{Q}$ be an elliptic curve and (\ref{we}) a Weierstrass equation for $E$. Siegel \cite{Sie01} proved that there are finitely many $P \in E(\mathbb{Q})$ with $B_P=1$. In \cite{MR2365225} it is shown that for fixed $l>1$, there are finitely many $P \in E(\mathbb{Q})$ with $B_P=z^l$ for some $z \in \mathbb{Z}$. Since their denominator is a perfect power, perhaps it is reasonable to give the following

\begin{defn} Call $P \in E(\mathbb{Q})$ \emph{power integral} if $B_P$ (as in (\ref{B_P})) is equal to a  perfect power.
\end{defn}

Note that $1$ is a perfect power and so power integral points can be thought of as a generalization of the integral points. A lot of work has been done to make Siegel's theorem effective \cite{MR0231783,MR0263742, MR735341, MR1328329, MR518817} and there are many techniques which can find all of the integral points for large classes of elliptic curves 
\cite{MR1305199, MR1713117, MR1291875, MR1986812}. For certain curves we are now able to find all of the power integral points.   
 
\subsection{Mordell curves}  Theorem \ref{2or3} can be strengthened considerably for Mordell curves.

\begin{thm} \label{Mor}
Let $D \in \mathbb{Z}$ be non-zero and $E: y^2=x^3+D$. There are finitely many perfect powers in an elliptic divisibility sequence $(B_m)$ whose first term is greater than $1$. As in Theorem \ref{2or3}, there is a bound depending on $D$ and $P$ for the possible prime exponents.
\end{thm}

By utilizing the proofs of the theorems above in a specific case we are able to find the Mordell curve of smallest conductor with non-zero rank and no power integral points.  

\begin{prop} \label{D=11}
The elliptic curve $E: y^2=x^3+11$ has no power integral points. 
\end{prop}

In the general case, allowing for integral points, we expect the following to hold.

\begin{con} \label{con1}
Let $D \in \mathbb{Z}$ be non-zero and $E: y^2=x^3+D$. For $l$ a sufficiently large prime, if $B_P$ (as in (\ref{B_P})) is an $l$th power then $B_P=1$.  
\end{con}

At the end of Section \ref{mdc} it is explained that Conjecture \ref{con1} would follow from the Frey-Mazur conjecture \cite{MR1357210}. 

\subsection{Congruent number curves}

A much studied class of elliptic curves is
the congruent number curves $E_N : y^2 = x^3-N^2x$, where $N \ge 1$ is an integer. Let $p$ be an odd prime and $a$, $b$ non-negative integers. For 
$N = 2^ap^b$, a simple algorithm for the determination of the integral points in $E_N(\mathbb{Q})$ has been given in  \cite{MR2219047} and~\cite{MR2219040}. In this case we are able to find all power integral points in $2E_N(\mathbb{Q})$; in fact they are all integral.

\begin{thm} \label{first}
Let $N = 2^ap^b$. If $P \in 2E_N(\mathbb{Q})$  is power integral then 
$N=2^a3^b$ and $P=(c^225,c^335)$, where $a,b$ are odd, $a \ge 3$ and $c= \pm 2^{(a-3)/2}3^{(b-1)/2}$.
\end{thm}

In Section \ref{red} Theorem \ref{first} is proven using Fermat's Last Theorem, due to Wiles \cite{MR1333035}, along with the first variants by  
Ribet \cite{MR1438112}, Darmon and Merel~\cite{MR1468926}. 

\begin{thm} \label{second} Let $N=2^ap$, where $a=0$ or $1$. Suppose that 
$P \in E_N(\mathbb{Q})$ has 
\[
x(P) \in   -{\mathbb{Q}^*}^2
\]
and
\[   
x(P)+N \in  p{\mathbb{Q}^*}^2. 
\]
Then there are no perfect powers in the elliptic divisibility sequence generated by $P$.
\end{thm}

Theorem \ref{second} is proven using Theorem \ref{first} along with an equation recently solved by Bennett, Ellenberg and Ng \cite{MR2646760}.  

\begin{exa} \label{exa8}
There are no perfect powers in the elliptic divisibility sequence generated by 
$(-(60/41)^2,-455700/41^3)$ on $E_5:y^2=x^3-25x$.
\end{exa}

Let $N=2^ap$, where $a=0$ or $1$. Points belonging to two cosets in $E_N(\mathbb{Q})/2E_N(\mathbb{Q})$ have been considered above. The remaining cases lead to equations which currently appear unresolvable in general. As the next example shows, there can be power integral points on $E_N$ which are not integral. However, in the example $N$ is equal to the odd terms of a sequence $(C_m)$ and, since these odd terms form an elliptic divisibility sequence, there are conjectured to be finitely many possibilities with $N$ prime  \cite{MR2548983}. This, along with Theorem~\ref{first} and Theorem \ref{second}, suggests that the number of power integral points in $E_N(\mathbb{Q})$ which are not integral could be uniformly bounded.   

\begin{exa} For $(-1,1)$ on $y^2=x^3-2x$ and $m$ odd write
\[
m(-1,1)=\left(-\frac{A_m^2}{B_m^2}, \frac{A_mC_m}{B_m^3} \right).
\]
We get a power integral point on $E_{C_m}: y^2=x^3-C_m^2x$ given by $x(P)=-(C_mA_m)^2/B_m^4$. Moreover, $C_m$ is prime for $m=3,7$ and $23$.
\end{exa}

\section{Properties of elliptic divisibility sequences}

In this section the required properties of elliptic divisibility sequences are collected.

\begin{lem} \label{edsp} Let $(B_m)$ be an elliptic divisibility sequence. 
\begin{enumerate}[(i)]
\item Let $p$ be a prime and $m_0$ be the smallest positive integer such that $p \mid B_{m_0}$. Then for every $m \in \mathbb{N}$, 
\[
p \mid B_m \iff m_0 \mid m.
\]
\item Let $p$ be an odd prime. For any pair $n,m \in \mathbb{N}$, if $\ord_p(B_n)>0$ then
\[
\ord_p(B_{mn})=\ord_p(B_n)+\ord_p(m).
\] 
\item For any pair $n,m \in \mathbb{N}$, 
if $2 \mid B_n$ then
\[ 
\ord_2(B_{mn}) = \ord_2(B_n)+\ord_2(m) 
\]
if $a_1$ is even and
\[
 \left| \ord_2(B_{mn}) -(\ord_2(B_n)+\ord_2(m)) \right| \le \epsilon
\] 
otherwise, where the constant $\epsilon$ depends only on $E$ and $P$.
\item For all $m,n \in \mathbb{N}$,
\[
\gcd(B_m,B_n)=B_{\gcd(m,n)}.
\]
\end{enumerate}
\end{lem}

\begin{proof}
See \cite{MR961918} and Section 4 in Chapter IV  of \cite{Marcothesis}. 
\end{proof}

\begin{rem} \rm Given a prime $p$, a rational point on (\ref{we}) can be reduced modulo $p$ to give a point defined over a finite field (see Chapter VII of \cite{MR2514094}). It follows that $m_0$ in Lemma~\ref{edsp} exists for every prime at which the generator of the elliptic divisibility sequence has non-singular reduction.    
\end{rem}

\begin{lem} \label{three} Assume that the given Weierstrass equation for $E$ has $a_1$ even. For a prime $p$ suppose that $m_0=p$ in Lemma \ref{edsp}. 
Write $m=p^em'$ where $p \nmid m'$. If $B_m$ is an $l$th power then so is 
$B_{m'}$. Moreover, $p \nmid B_{m'}$.       
\end{lem}

\begin{proof}
By Lemma \ref{edsp}, if a prime $q$ divides $B_{m'}$ then $q \ne p$ and
\[
\ord_q(B_m)=\ord_q(B_{m'})+\ord_q(p^e)=\ord_q(B_{m'})
\]
so the result follows.
\end{proof}

\begin{defn}
A prime $p | B_m$ such that $p \nmid B_{m'}$ for any $m' < m$ is called a \emph{primitive divisor}.
\end{defn}

\begin{thm}[Silverman] \label{prim} For all but finitely many $m \in \mathbb{N}$, 
$B_m$ has a primitive divisor. Moreover, if $B_m$ does not have a primitive divisor then $m$ is bounded by an effectively computable constant which depends only on $E$ and $P$.
\end{thm}

\begin{proof} See \cite{MR961918} or Chapter IV of \cite{Marcothesis}.   \end{proof}

\begin{rem} \rm For certain minimal Weierstrass equations the number of terms without a primitive divisor has been uniformly bounded (see \cite{MR2301226, MR2605536, IngrSilv}).
\end{rem}   

\section{The modular approach to Diophantine equations}

For a more thorough exploration see \cite{Sanderthesis} and Chapter 15 in \cite{MR2312338}. As is conventional, in what follows all newforms shall have weight $2$ with a trivial character at some level $N$ and shall be thought of as a $q$-expansion
\[
f=q+\sum_{n \ge 2}c_nq^n, 
\]
where the field $K_f=\mathbb{Q}(c_2,c_3,\cdots)$ is a totally real number field. The coefficients $c_n$ are 
algebraic integers and $f$ is called \emph{rational} if they all belong to $\mathbb{Z}$. For a given level $N$, the number of newforms is finite. The modular symbols algorithm \cite{MR1628193}, implemented on $\mathtt{MAGMA}$ \cite{MR1484478} by William Stein, shall be used to compute the newforms at a given level. 

\begin{thm}[Modularity Theorem]
Let $E/\mathbb{Q}$ be an elliptic curve of conductor $N$. Then there exists a newform $f$ of level $N$ such that $a_p(E)=c_p$ for all primes 
$p \nmid N$, where $c_p$ is $p$th coefficient of $f$ and 
$a_p(E)=p+1-\#E(\mathbb{F}_p)$.  
\end{thm}

\begin{proof}
This is due to Taylor and Wiles \cite{MR1333036, MR1333035} in the semi-stable case. The proof was completed by Breuil, Conrad, Diamond and Taylor \cite{MR1839918}.     
\end{proof}

The modularity of elliptic curves over $\mathbb{Q}$ can be seen as a converse to   

\begin{thm}[Eichler-Shimura] Let $f$ be a rational newform of level $N$. There exists an elliptic curve $E/\mathbb{Q}$ of conductor $N$ such that 
$a_p(E)=c_p$ for all primes $p \nmid N$, where $c_p$ is the $p$th coefficient of $f$ and $a_p(E)=p+1-\#E(\mathbb{F}_p)$.
\end{thm}

\begin{proof}
See Chapter $8$ of \cite{MR2112196}.
\end{proof}

Given a rational newform of level $N$, the elliptic curves of conductor $N$ associated to it via the Eichler-Shimura theorem shall be computed using $\mathtt{MAGMA}$.  

\begin{prop} \label{levellow} Let $E/\mathbb{Q}$ be an elliptic curve with conductor $N$ and minimal discriminant $\Delta_{\min}$.
Let $l$ be an odd prime and define
\[ 
N_0(E,l):=N/\mathop{\prod_{{\textrm{primes } p \mid \mid N}}}_{l \mid \ord_p(\Delta_{\min})} p. 
\]
Suppose that the Galois representation
\[ 
\rho_l^E: \gal(\bar{\mathbb{Q}}/\mathbb{Q}) \to \aut(E[l])
\]
is irreducible. Then there exists a newform $f$ of level $N_0(E,l)$. Also there exists  a prime $\mathcal{L}$ lying above $l$ in the ring of integers $\mathcal{O}_f$ defined by the coefficients of $f$ such that  
\[
c_p \equiv \left\{ \begin{array}{ll} a_p(E)  \mod \mathcal{L} & \textrm{ if } p \nmid lN, \\ \pm(1+p) \mod \mathcal{L} & \textrm{ if } p \mid \mid N \textrm{ and } p \nmid lN_0, \end{array} \right.
\]
where $c_p$ is the $p$th coefficient of $f$. Furthermore, if $\mathcal{O}_f=\mathbb{Z}$ then
\[
c_p \equiv \left\{ \begin{array}{ll} a_p(E)  \mod l & \textrm{ if } p \nmid N, \\ \pm(1+p) \mod l & \textrm{ if } p \mid \mid N \textrm{ and } p \nmid N_0. \end{array} \right.
\]    
\end{prop}

\begin{proof}
This arose from combining modularity with level-lowering results by Ribet~\cite{MR1047143, MR1265566}. The strengthening in the case $\mathcal{O}_f=\mathbb{Z}$ is due to Kraus and Oesterl{\'e}~\cite{MR1166121}.   
A detailed exploration is given, for example, in Chapter 2 of \cite{Sanderthesis}. 
\end{proof}

\begin{rem} \rm Let $E/\mathbb{Q}$ be an elliptic curve with conductor $N$.
Note that the exponents of the primes in the factorization of $N$ are uniformly bounded (see Section 10 in Chapter IV of \cite{MR1312368}). In particular, only primes of bad reduction divide $N$ and if $E$ has multiplicative reduction at $p$ then $p \mid \mid N$. 
\end{rem}

\begin{cor} \label{lbound} Keeping the notation of Proposition \ref{levellow}, if $p$ is a prime such that $p \nmid lN_0$ and  $p \mid N$ then
\[
l < (1+\sqrt{p})^{2[K_f: \mathbb{Q}]}.
\]    
\end{cor}

\begin{proof}
See Theorem 37 in \cite{Sanderthesis}.
\end{proof}

Applying Proposition \ref{levellow} to carefully constructed Frey curves has led to the solution of many Diophantine problems. The most famous of these is Fermat's Last theorem \cite{MR1333035} but there are now constructions for other equations and we shall make use of those described below.

\subsection{Recipes for Diophantine equations with signature $(l,l,3)$}
Consider the equation
\[
Ax^l+By^l=Cz^3,
\]
with non-zero pairwise coprime terms and $l \ge 5$ prime. Assume any prime $q$ satisfies $\ord_q(A)<l$, $\ord_q(B)<l$ and $\ord_q(C)<3$. Without lost of generality also assume that $Ax \nequiv 0 \mod 3$ and $By^l \nequiv 2 \mod 3$.  
Construct the Frey curve
\[
E_{x,y}: Y^2+3CzXY+C^2By^lY=X^3. 
\]
\begin{thm}[Bennett, Vatsal and Yazdani \cite{BenDah}] \label{reci}
The conductor $N_{x,y}$ of $E_{x,y}$ is given by
\[
N_{x,y}=3^{\alpha}\rad_3(ABxy)\rad_3(C)^2,
\]
where
\[
\alpha= \left\{ \begin{array}{ll} 2 & \textrm{if } \textrm{ } 9 \mid (2+C^2By^l-3Cz), \\
3 & \textrm{if } \textrm{ } 3\mid \mid (2+C^2By^l-3Cz), \\
4 & \textrm{if } \textrm{ } \ord_3(By^l)=1, \\
3 & \textrm{if } \textrm{ } \ord_3(By^l)=2, \\
0 & \textrm{if } \textrm{ } \ord_3(By^l)=3, \\
1 & \textrm{if } \textrm{ } \ord_3(By^l) \ge 4, \\
5 & \textrm{if } \textrm{ } 3 \mid C.
\end{array} \right.  
\]
Suppose that $E_{x,y}$ does not correspond to one of the
equations
\begin{eqnarray*}
1 \cdot 2^5+27 \cdot (-1)^5 &=& 5 \cdot 1^3, \\ 
1 \cdot 2^7+3 \cdot (-1)^7 &=& 1 \cdot 5^3, \\
2\cdot 1^2+27\cdot (-1)^5&=&25 \cdot (-1)^3, \textrm{ or } \\
2 \cdot 1^7 + 3 \cdot (-1)^7&=&(-1)^3. 
\end{eqnarray*}
Then there exists a newform of level
\[
N_0=3^{\beta}\rad_3(AB)\rad_3(C)^2, 
\]
where
\[
\beta= \left\{ \begin{array}{ll} 2 & \textrm{if } \textrm{ } 9 \mid (2+C^2By^l-3Cz), \\
3 & \textrm{if } \textrm{ } 3\mid \mid (2+C^2By^l-3Cz), \\
4 & \textrm{if } \textrm{ } \ord_3(By^l)=1, \\
3 & \textrm{if } \textrm{ } \ord_3(By^l)=2, \\
0 & \textrm{if } \textrm{ } \ord_3(B)=3, \\
1 & \textrm{if } \textrm{ } \ord_3(By^l)\ge 4 \textrm{ and } \ord_3(B) \ne 3, \\
5 & \textrm{if } \textrm{ } 3 \mid C.
\end{array} \right.  
\]
\end{thm}  
    
\subsection{Frey-Hellegouarch curves for Klein forms} \label{FH}

Let $E$ be an elliptic curve defined over $\mathbb{Q}$ with Weierstrass coordinate functions $x,y$. For any integer $n \in \mathbb{Z}$, the $n$th \emph{division polynomial} of $E$ is the polynomial 
$\psi_n \in \mathbb{Q}[x,y] \subset \mathbb{Q}(E)$ as given on p. 39 of \cite{MR1771549}. In particular,
\begin{eqnarray*}
\psi_2^2&=&4x^3 + b_2x^2 + 2b_4x + b_6, \\
\psi_3&=&3x^4 + b_2x^3 + 3b_4x^2 + 3b_6x + b_8,
\end{eqnarray*}
$\psi_n^2 \in \mathbb{Q}[x]$ and there exists $\theta_n \in \mathbb{Q}[x]$ such that
\begin{equation} \label{nQ}
[n]x=\frac{\theta_n}{\psi_n^2}.
\end{equation}
\begin{defn} \label{K} Associate to $\psi_2^2(x)$ and $\psi_3(x)$ the homogeneous polynomials:
\[
K^E_n(x,y)= \left\{ \begin{array}{ll} \psi_2^2(x/y)y^3 & \textrm{ for } $n=2$, \\ \psi_3(x/y)y^4 & \textrm{ for } $n=3$.  \end{array} \right. 
\]
\end{defn}
The notion of a Klein form arose from Klein's  classification \cite{MR0080930} of the finite subgroups of $\aut_{\bar{\mathbb{Q}}}(\mathbb{P}_1)$. For our purposes it is enough to note that any separable cubic binary form in $\mathbb{Q}[x,y]$ is a Klein form and that a separable quartic
\[
\alpha_0x^4+\alpha_1x^3y+\alpha_2x^2y^2+\alpha_3xy^3+\alpha_4y^4 \in \mathbb{Q}[x,y]
\]
is a Klein form precisely when
\begin{equation} \label{transvectant}
12\alpha_0\alpha_4-3\alpha_1\alpha_3+\alpha_2^2= 0.
\end{equation}
\begin{lem} Let $E$ be an elliptic curve defined over $\mathbb{Q}$. Then $K^E_2(x,y)$ and $K^E_3(x,y)$ are Klein forms.
\end{lem}

\begin{proof}
Since the multiplication by $n$ map is separable (see Chapter III of \cite{MR2514094}), $K^E_n(x,y)$ is separable. A small calculation checks that the coefficients of $K^E_3(x,y)$ satisfy (\ref{transvectant}).
\end{proof}
For $S$ a fixed finite set of primes, let
\[
\mathbb{Z}_S := \{ x \in \mathbb{Q} : \ord_p(x) \ge 0 \textrm{ for all } p \notin S   \}
\]
and let $\mathbb{Z}_S^*$ be the set of units in $\mathbb{Z}_S$. Let 
$F$ be a Klein form of degree $k \in \{3,4,6,12\}$ 
($k=3$ or $4$ is enough for our purposes). The \emph{index} of $F$ is $n=6-12/k$. Denote by 
$\Delta_F$ the discriminant of $F$ and let $S_F$ be the set of primes which divide $n\Delta_F$. In \cite{BenDah, Sanderthesis} Bennett and Dahmen construct a Frey-Hellegouarch curve for the Diophantine equation
\begin{equation} \label{power}
F(A,B)=uC^l, 
\end{equation}
where $\gcd(A,B)=1$, $C \ne 0$, $l$ is prime and $u \in \mathbb{Z}_{S_F}^*$. Define
\[
H(x,y)=\frac{1}{(k-1)^2} \left| \begin{array}{cc} F_{xx} & F_{xy} \\ F_{xy} & F_{yy} \end{array} \right|
\]
and the Jacobian determinant of $F$ and $H$ by
\[
G(x,y)=\frac{1}{k-2} \left| \begin{array}{cc} F_{x} & F_{y} \\ H_{x} & H_{y} \end{array} \right|,
\]
where $F_x$, $F_y$, etc, refer to corresponding partial
derivatives. Then
\[
4H(A,B)^3+G(A,B)^2=d_nF(A,B)^n,
\]
where $d_2=-27\Delta_F$ and $d_3=2^8\sqrt{-\Delta_F/27}$ are integers. So
\[
E_{A,B}: Y^2=X^3+3H(A,B)X+G(A,B)
\]
has discriminant $-2^4 \cdot 3^3d_nF(A,B)^2$. 
\begin{prop}[\cite{BenDah}] \label{t}
There exists $t \in \{\pm 1, \pm 3\}$ such that for all primes $p \notin S_F$ 
we have that the quadratic twist
\[
E_{A,B}^{(t)}: Y^2=X^3+3H(A,B)t^2X+G(A,B)t^3
\]
is semistable at $p$ and 
\[
\ord_p(\Delta_{min}(E_{A,B}^{(t)}))=n\ord_p(F(A,B)).
\] 
\end{prop}
\begin{proof}
This is Proposition 4.2 of \cite{BenDah}.
\end{proof}

\begin{prop}[\cite{BenDah}] \label{newform} Let $l>163$ in (\ref{power}) and let $t$ be as in Proposition \ref{t}. Denote by $N_{A,B}$ the conductor of $E_{A,B}^{(t)}$. 
Then the Galois representation   
\[
\rho_l^{A,B}: \gal(\bar{\mathbb{Q}}/\mathbb{Q}) \to \aut(E_{A,B}^{(t)}[l])
\]
is modular of level
\begin{equation} \label{N_0}
N_0=\prod_{p \mid n\Delta_F}p^{\ord_p(N_{A,B})}. 
\end{equation}
In particular, there exists a newform $f$ of level $N_0$.    
\end{prop}

\begin{proof}
This is Proposition 8.1 in \cite{BenDah}.
\end{proof}

\subsection{A similar Frey curve for cubic forms}
The Frey curve already given in Section~\ref{FH} can be seen as sufficient; however, for ease of reference we give a construction from~\cite{MR2406491}.

Let
\[
F(x,y)=t_0a^3+t_1^2y+t_2xy^2+t_3y^3 \in \mathbb{Z}[x,y]
\]
be a separable cubic binary form. 
In \cite{MR2406491} a Frey curve is given for the Diophantine equation
\begin{equation} \label{cub}
F(a,b)=dc^l,
\end{equation}
where $\gcd(a,b)=1$, $d \in \mathbb{Z}$ is fixed and $l \ge 7$ is prime. 
Define a Frey curve $E_{a,b}$ by
\begin{equation} \label{BillFrey}
E_{a,b}: y^2=x^3+a_2x^2+a_4x+a_6,
\end{equation}
where
\begin{eqnarray*}
a_2 &=& t_1a -t_2b, \\
a_4 &=& t_0t_2a^2 +(3t_0t_3 -t_1t_2)ab + t_1t_3b^2, \\
a_6 &=& t_0^2t_3a^3-t_0(t_2^2-2t_1t_3)a^2b+t_3(t_1^2-2t_0t_2)ab^2-t_0t_3^2b^3.
\end{eqnarray*}
Then $E_{a,b}$ has discriminant $16 \Delta_F F(a,b)^2$. Consider the Galois representation
\[
\rho_l^{a,b} : \gal(\bar{\mathbb{Q}}/\mathbb{Q}) \to \aut(E_{a,b}[l]).
\]
\begin{thm}[\cite{MR2406491}] \label{Bill} Let $S$ be the set of primes dividing $2d\Delta_F$. 
There exists a constant $\alpha(d,F) \ge 0$ such that if $l>\alpha(d,F)$ 
and $c \ne \pm 1$ then:
\begin{itemize}
\item the representation $\rho_l^{a,b}$ is irreducible;
\item at any prime $p \notin S$ dividing $F(a,b)$ the equation (\ref{BillFrey}) is minimal, the elliptic curve $E_{a,b}$ has multiplicative reduction and $l \mid \ord_p(\Delta_{min}(E_{a,b}))$.   
\end{itemize} 
\end{thm}

\begin{proof}
This is Theorem 2.3 and Lemma 2.4 in \cite{MR2406491}.
\end{proof}
     
\section{Proof of Theorem \ref{2or3}}

\begin{proof}[Proof of Theorem \ref{2or3}]
Let $n=2$ or $3$ and let $S$ be the set of primes dividing $n \Delta_E$.
Assume that $B_m$ is an $l$th power. Note that by Theorem 1.1 in \cite{MR2365225} it is enough to bound $l$ in terms of $E$ and $P$. To do this we shall derive an equation of the form (\ref{power}) and prove the existence of a prime divisor $p_0$ to which Corollary \ref{lbound} can be applied.    

Using Theorem \ref{prim}, fix  $e_0 \ge 1$  
such that
\begin{itemize}
\item $B_{n^{e_0}}$ is divisible by a prime $p_0 \nmid n\Delta_E$,
\item $p_0 \nmid B_{n^e}$ for all $0 \le e <e_0$. 
\end{itemize}
Note that $e_0$ does not depend on $m$. From Lemma \ref{edsp}, since $\ord_n(B_1)>0$,
\[
\ord_n(B_{m}) -(\ord_n(B_1)+\ord_n(m)) = O(1).
\]
Hence, since $l \mid \ord_n(B_{m})$, we can assume that $l$ is large enough so that 
\begin{equation} \label{e_0}
\ord_n(m) \ge e_0.
\end{equation}   
For $Q \in E(\mathbb{Q})$, using (\ref{nQ}) gives
\begin{equation} \label{2Q}
\frac{A_{nQ}}{B_{nQ}^2}=\frac{\theta_n(A_Q/B_Q^2)}{\psi_n^2(A_Q/B_Q^2)}=
\frac{B_Q^{2n^2}\theta_n(A_Q/B_Q^2)}{B_Q^2 \psi_n^2(A_Q/B_Q^2)B_Q^{2(n^2-1)}},
\end{equation}
where 
\[
\psi_n^2(A_Q/B_Q^2)B_Q^{2(n^2-1)}=\left\{ \begin{array}{ll} K_2^E(A_Q,B_Q^2) & 
\textrm{ if } n=2, \\ (K_3^E(A_Q,B_Q^2))^2 & \textrm{ if } n=3   \end{array} \right.
\]
(see Definition \ref{K}). 
Since $\theta_n$ is monic and the leading coefficient of $\psi_n^2$ is $n^2$, 
$B_Q$ is coprime with the numerator of (\ref{2Q}) and if $B_{nQ}$ is an $l$th power then $B_Q$ is a power of $n$ multiplied by an $l$th power. Write $m=n^{\ord_n{m}}m'$ with $n \nmid m'$. From (\ref{e_0}) it follows that 
$B_{n^{e_0}m'}$ is a power of $n$ multiplied by an $l$th power. 
Write $Q=n^{e_0-1}m'P$ then $n^{e_0}m'P=nQ$. The primes which divide the numerator and the denominator of (\ref{2Q}) also divide the discriminant $\Delta_E$ (see~\cite{MR1185022}). So
\begin{equation} \label{K_n}
K_n^E(A_Q,B_Q^2)=uC^l,
\end{equation}
where $u \in \mathbb{Z}_S^*$. Moreover, $p_0 \nmid B_Q$ 
(since $\gcd(B_Q,B_{n^{e_0}})=B_{n^{e_0-1}}$) and so 
$C \in \mathbb{Z}$ is divisible by $p_0$. In characteristic away from $n$ the multiplication by $n$ map is separable (see Chapter III of \cite{MR2514094}) so the set of primes which divide the discriminant of $K_n^E$ is equal to $S$. Applying Proposition \ref{newform} shows that there exists a newform $f$ of level $N_0$ (as in (\ref{N_0})). It follows that there are finitely many choices for $f$. We have $p_0 \nmid lN_0$ and 
$p_0 \mid N_{A_Q,B_Q}$ (see Proposition \ref{t}) so Corollary \ref{lbound} bounds $l$. \end{proof}

\begin{rem} \rm
Note that $K_n^E(A_Q,B_Q^2)$, as in (\ref{K_n}), does not belong to $\mathbb{Z}_S^*$ so Theorem 8.1 in~\cite{BenDah} along with Silverman's
primitive divisor theorem proves the existence of an effectively computable bound for $l$ which depends only on $E$ and $P$. However, keeping in mind that  $p_0 \nmid lN_0$, in practice a much better bound is obtained by computing the newforms at level $N_0$ and applying Proposition \ref{levellow} directly.  
\end{rem}

\begin{rem} \label{Sext} \rm Let $S$ be a finite set of fixed primes and let $(B_m)$ be an elliptic divisibility sequence whose first term is divisible by $2$ or $3$ . The results in Section \ref{FH} hold with the primes in $S$ added to $S_F$. Using this the proof above can be extended to show that there are finitely terms in  $(B_m)$ equal to a perfect power multiplied by an $S$-unit.        
\end{rem}

\section{The Mordell curves $y^2=x^3+D$} \label{mdc}

\begin{proof}[Proof of Theorem \ref{Mor}]
Write $D=d^2D'$, where $D'$ is square free. Suppose that $P \in E(\mathbb{Q})$ with $x(P) \ne 0$ and $B_P=z^l$ for some prime $l$. Factorizing over $K=\mathbb{Q}(\sqrt{D'})$,
\[
A_P^3=C_P^2-Dz^{6l}=(C_P+d\sqrt{D'}z^{3l})(C_P-d\sqrt{D'}z^{3l}). 
\]

If $D'=1$ then $C_P+dz^{3l}=ua^3$ and $C_P-dz^{3l}=vb^3$, where $a,b \in \mathbb{Z}$ are coprime, $u,v$ divide $2d$ and $uv$ is a cube. Subtracting the two factors gives
\begin{equation} \label{D'=1}
2dz^{3l}=ua^3-vb^3.
\end{equation}

In general the ring $\mathcal{O}_{KT}$ is a principal ideal domain for some finite set $T$ of prime ideals. Include in $T$ the primes in $\mathcal{O}_K$ dividing 
$2d\sqrt{D'}$. Using Dirichlet's unit theorem, $\mathcal{O}_{KT}^*/{\mathcal{O}_{KT}^*}^3$ is a finite set. Hence, if $D' \ne 1$ then 
\[
C_P+d\sqrt{D'}z^{3l}=(u+\sqrt{D'}v)(a+b\sqrt{D'})^3,
\]
where $a,b \in \mathbb{Z}$ are coprime and there are finitely many choices for $u,v \in \mathbb{Q}$. Subtracting the two conjugate factors gives
\begin{equation} \label{1}
dz^{3l}=va^3+3ua^2b+3D'vab^2+uD'b^3. 
\end{equation}

Now suppose that $(B_m)$ is an elliptic divisibility sequence generated by a point on $E$. Multiplying through by the denominators of $u,v$ in (\ref{D'=1}) or (\ref{1}) gives an equation
\[
F(a,b)=dc^l
\] 
as in (\ref{cub}) with $c^l=B_m^3$. Note that $u$ and $v$ are non-zero in (\ref{D'=1}) and at least one of $u,v$ is non-zero in (\ref{1}); it follows that the cubic forms considered are separable. Construct a Frey curve 
$E_{a,b}$ as in (\ref{BillFrey}). Let $S$ be the set of primes dividing 
$2d\Delta_F$.
  
Assume that $n \mid B_1$ and $n>1$ is prime. Using the Siegel-Mahler theorem about finiteness of $S$-integral points on elliptic curves, fix $e_0 \ge 1$ such that $B_{n^{e_0}}$ is divisible by a prime $p_0 \notin S$. Note that $e_0$ does not depend on $m$. From Lemma \ref{edsp}, since $\ord_n(B_1)>1$,
\[
\ord_n(B_{m}) -(\ord_n(B_1)+\ord_n(m)) = O(1).
\]
Hence, since $l \mid \ord_n(B_{m})$, we can assume $l$ is large enough so that 
\[
\ord_n(m) \ge e_0.
\]    
Then $B_{n^{e_0}} \mid B_m$ and, in particular, $p_0 \mid B_m$. Applying 
Theorem \ref{Bill} and Proposition \ref{levellow} with $p=p_0$ gives that 
$l$ is bounded. (Note that $p_0$ divies the conductor of the Frey curve but
not the level of the newform.) The finiteness claim follows from Theorem 1.1 in \cite{MR2365225}. \end{proof}

Let $D \in \mathbb{Z}$ be square-free. For $P \in 2E(\mathbb{Q})$ write $P=2Q$. Using the duplication formula,
\begin{equation} \label{2Q2}
\frac{A_P}{B_P^2}=\frac{A_Q(A_Q^3-8DB_Q^6)}{4B_Q^2(A_Q^3+DB_Q^6)}=\frac{A_Q(A_Q^3-8DB_Q^6)}{4B_Q^2C_Q^2}.
\end{equation}
Any prime dividing $C_Q$ and $A_Q^3-8DB_Q^6$ also divides $3D$. Suppose that $B_P$ is an $l$th power and that $B_Q$ is even. Since $D$ is square-free, $\gcd(A_Q,C_Q)=1$ so only 
$3$ can divide both $C_Q$ and the numerator of (\ref{2Q2}). If $\gcd(3,C_Q)=1$ then $C_Q$ and $2B_Q$ must be an $l$th powers.

\begin{proof}[Proof of Proposition \ref{D=11}]
Note that $E(\mathbb{Q})=\left< (-7/4,19/8) \right>$. Let $P=m(-7/4,19/8)$ for some $m \ge 1$ and denote $B_P$, as in (\ref{B_P}), by $B_m$. Assume that 
$B_P$ is an $l$th power. Using Lemma \ref{three} we can assume that $3 \nmid B_P$ and $3 \nmid m$. From Lemma \ref{edsp},
\begin{equation} \label{ord}
\ord_2(B_m)=\ord_2(B_1)+\ord_2(m)=1+\ord_2(m) \ge l
\end{equation}
so $m$ is even. Thus $P=2Q$ for some $Q \in E(\mathbb{Q})$. 
By (\ref{2Q2}) it follows that $C_Q$ and $2B_Q$ are $l$th powers.
 
\begin{lem} \label{tec}
If $l>2$ then $13, 19$ and $619$ divide $B_Q$. Also $7 \mid A_Q$ but 
$7 \nmid B_QC_Q$.   
\end{lem}
\begin{proof}
If $l>2$ then (\ref{ord}) gives that $m=4m'$ thus $Q=2m'(-7/4,19/8)$ for some 
$m' \ge 1$ so $B_2 \mid B_Q$ and, in particular, $19 \mid B_Q$. 
Using Lemma \ref{edsp} again,
\[
\ord_{19}(B_Q)=\ord_{19}(B_2)+\ord_{19}(m')=1+\ord_{19}(m')=l
\]
so $\ord_{19}(m')>0$ and, in particular, $13 \mid B_Q$. 
Similarly, $B_{13} \mid B_Q$ so $619 \mid B_Q$.

Since $7 \mid B_3$ and $\gcd(B_P,B_3)=2$, 
we have that $7 \nmid B_P$ so, from (\ref{2Q2}), $7 \nmid B_QC_Q$. 
Reducing the equation $C_Q^2-11B_Q^6=A_Q^3$ modulo $7$ shows that 
$A_Q \equiv 0 \mod 7$.    \end{proof}       

Assume that $l \ge 5$. Consider the $(l,l,3)$ triple given by
\[
C_Q^2-11B_Q^6=A_Q^3, 
\]
where the three terms are pairwise coprime. As in Theorem \ref{reci} construct a Frey Curve
\[
E_Q: Y^2+3A_QXY-11B_Q^6Y=X^3
\]
with conductor $N_Q=3^{\alpha} \cdot 11\rad_3(C_QB_Q)$, where 
$\alpha=2$ or $3$. The Galois representation 
$\rho_l^{E_Q}: \gal(\bar{\mathbb{Q}}/\mathbb{Q}) \to \aut (E_Q[l])$ arises from a cuspidal newform $f$ of weight $2$ and level $N_0=2\cdot 3^{\alpha} \cdot 11$. This newform is one of
\begin{eqnarray*}
f_1 &=& q-q^2+q^4+2q^7-q^8+q^{11}+\cdots, \\
f_2 &=& q-q^2+q^4+4q^5-2q^7-q^8-4q^{10}-q^{11} + \cdots, \\
f_3 &=& q - q^2 + q^4 - 2q^5 - 4q^7 - q^8 + 2q^{10} + q^{11} + \cdots, \\
f_4 &=& q + q^2 + q^4 + 2q^7 + q^8 - q^{11} + \cdots, \\
f_5 &=& q + q^2 + q^4 + 2q^7 + q^8 + q^{11}+ \cdots,
\end{eqnarray*}
for $\alpha=2$ or (up to conjugacy) one of
\begin{eqnarray*}
f_6 &=& q - q^2 + q^4 - 2q^5 + q^7 - q^8 + 2q^{10} - q^{11} + \cdots \\
f_7 &=& q - q^2 + q^4 + q^5 + 4q^7 - q^8 - q^{10} - q^{11} + \cdots \\
f_8 &=& q - q^2 + q^4 - 3q^5 - 4q^7 - q^8 + 3q^{10} - q^{11} + \cdots \\
f_9 &=& q - q^2 + q^4 - 2q^5 - q^7 - q^8 + 2q^{10} + q^{11} + \cdots \\
f_{10} &=& q + q^2 + q^4 + 2q^5 - q^7 + q^8 + 2q^{10} - q^{11} + \cdots \\
f_{11} &=& q + q^2 + q^4 - q^5 + 4q^7 + q^8 - q^{10} + q^{11} + \cdots \\
f_{12} &=& q + q^2 + q^4 + 2q^5 + q^7 + q^8 + 2q^{10} + q^{11} + \cdots \\
f_{13} &=& q + q^2 + q^4 + 3q^5 - 4q^7 + q^8 + 3q^{10} + q^{11} + \cdots \\
f_{14} &=& q - q^2 + q^4 + \theta q^5 + 2q^7 - q^8 - \theta q^{10} + q^{11} + \cdots \\
f_{15} &=& q + q^2 + q^4 + \theta q^5 + 2q^7 + q^8 + \theta q^{10} - q^{11} + \cdots
\end{eqnarray*}
for $\alpha=3$, where the last two are defined over a quadratic number field. Applying Proposition \ref{levellow} with $p=13$, $19$, $619$ gives $l=5$ if 
$f=f_2$ and $l<5$ (a contradiction) otherwise. If $f=f_2$ then applying Proposition 4.2 in \cite{MR2098394} with $p=5$ gives a contradiction; 
note that $5$ is a prime of good reduction and $f_2$ is rational so the restriction 
$l \ne 5$ in the Proposition can be removed. 
 
To eliminate the possibilities of $l=2$ or $3$ consider the parameterizations given in (\ref{1}) with $D=11$. Then $K=\mathbb{Q}(\sqrt{11})$ and 
$\mathcal{O}_K=\mathbb{Z}[\sqrt{11}]$ is a principal ideal domain with fundamental unit $10+3\sqrt{11}$. Also 
$2\sqrt{11}=(10-3\sqrt{11})\sqrt{11}(3+\sqrt{11})^2$. It follows that 
$(u,v)=(1,0), (10,3), (10,-3), (199,60)$ or $(199,-60)$.   

If $(u,v)=(1,0)$ then $C_Q=a(a^2+33b^2)$ and $B_Q^3=b(3a^2+11b^2)$,
where $b$ is even (since $4 \mid B_Q^3$). Since $33 \nmid C_Q$, $a$ is an $l$th power. Write $a=C^l$. Since $2 \mid b$ and $3 \nmid b$, $b=2^{3(l-1)}B^{3l}$.
Write $a^2+33b^2=\bar{C}$ and $3a^2+11b^2=\bar{B}$. Then
\begin{eqnarray}
3C^{2l}+2^{6(l-1)}11B^{6l}&=&\bar{B}^{3l}, \\
C^{2l}+2^{6(l-1)}33B^{6l}&=&\bar{C}^l, \\
\bar{B}^{3l}-3\bar{C}^l&=&2^{6l-3}11B^{6l}, \label{j3} \\
\bar{C}^l-3\bar{B}^{3l} &=&-8C^{2l}, \label{j4}
\end{eqnarray}
where the terms in each of the ternary equations are nonzero and pairwise coprime. If $l=2$ then (\ref{j3}) becomes $-3\bar{C}^2+(\bar{B}^2)^3-2^311(2B^2)^6=0$ and Proposition 6.5.9. in 
\cite{MR2312337} gives a rational point on the elliptic curve given by
$Y^2=X^3-2376$, but there are no such points. If $l=3$ then (\ref{j4}) 
becomes $\bar{C}^3+3(-\bar{B}^3)^3+(2C^2)^3=0$ and Proposition 6.4.14 in \cite{MR2312337} gives a rational point with non-zero coordinates on the elliptic curve given by $Y^2=X^3+144$, but there are no such points.     
 
For the other parameterizations details are given only for $(u,v)=(10,3)$. The other cases are similar. Assume that $(u,v)=(10,3)$. 
Then $A_Q=a^2-11b^2$,
\[
B_Q^3=3a^3+30a^2b+99ab^2+110b^3
\]
and 
\[
C_Q=10a^3+99a^2b+330ab^2+363b^3.
\]
Suppose that $l=2$. Since $C_Q$ and $2B_Q$ are squares, multiplying the
two expressions gives a rational point on the hyperelliptic curve
\[
F: Y^2=60X^6 + 1194X^5 + 9900X^4 + 43780X^3+108900X^2+144474X + 79860.
\]
But computations implemented in $\mathtt{MAGMA}$ confirm that the Jacobian of $F$ has rank $0$ and, via the method of Chabauty, $F(\mathbb{Q})$ is empty.
Finally, suppose that $l=3$. By Lemma \ref{tec}, $A_Q \equiv 0 \mod 7$. 
Hence, $a/b \equiv 2 \textrm{ or } 5 \mod 7$. 
Substituting these in the parametrization of $B_Q^3$ shows that 
$a/b \equiv 5 \mod 7$, but this cannot be a solution if $C_Q$ is a cube. This completes the proof of Proposition \ref{D=11}. 
\end{proof} 

By (\ref{D'=1}) and (\ref{1})  we see that Conjecture \ref{con1} would follow from 

\begin{con}[\cite{MR2406491}] \label{2.1} Let $F$ be a seperable homogenous cubic binary form with integer coefficients, $d$ a fixed integer $\ge 1$ and $l$ a prime number. There exists a constant $C_{d,F}>0$ depending only on $d$ and $F$ such that if $l>C_{d,F}$ and 
\[
F(a,b)=dc^l 
\]
with $\gcd(a,b)=1$ then $c=\pm 1$.
\end{con}

In \cite{MR2406491} it is explained that Conjecture \ref{2.1} would follow from the Frey-Mazur conjecture.

\begin{rem} \rm
A more direct Frey curve for $C_P^2=A_P^3+DB_P^6$ with $B_P$ an $l$th power is
\[
E_P: Y^2=X^3-3A_PX+2C_P.
\]
However, parametrizing as above highlights the connection with cubic binary forms and, as in the proof of Proposition \ref{D=11}, helps resolve specific cases. 
\end{rem}
          
\section{The congruent number curves $y^2=x^3-(2^ap^b)^2x$} \label{red}

Let $E_N$ be the elliptic curve given by $y^2=x^3-N^2x$, where $N$ is 
a congruent number. For a non-torsion point $P \in E_N(\mathbb{Q})$ there exists non-zero integers 
$z_1,z_2,z_3$ so that
\begin{eqnarray*}
A_P&=&\alpha_1 z_1^2, \\
A_P+NB_P^2&=&\alpha_2 z_2^2, \\
A_P-NB_P^2&=&\alpha_3 z_3^2,
\end{eqnarray*}
where the $\alpha_i$ are square free. Note that $\alpha_1 \mid N$, $\alpha_2 \mid 2N$ and $\gcd(z_1^2,z_2^2) \mid N$. So, in particular, $\gcd(z_1,z_2)=1$ if $N$ is square free. We have
\begin{eqnarray}
\alpha_2z_2^2-\alpha_1z_1^2 &=& NB_P^2; \label{e1} \\
\alpha_1z_1^2-\alpha_3z_3^2 &=& NB_P^2; \label{e3} \\
\alpha_2z_2^2-\alpha_3z_3^2 &=& 2NB_P^2; \label{e4} \\ 
2\alpha_1z_1^2-\alpha_2z_2^2&=&\alpha_3z_3^2. \label{e2}
\end{eqnarray}

\begin{thm}
Suppose that $l$ is an odd prime, $r$ is a non-negative integer and $U,V,W$ are non-zero pairwise coprime integers with
\begin{equation} \label{lem}
U^l+2^rV^l+W^l=0.
\end{equation}
Then $r=1$ and $(U,V,W)=\pm(-1,1,-1)$.
\end{thm}

\begin{proof}
The result is due to Wiles \cite{MR1333035} for $r=0$, Ribet \cite{MR1438112}  for $r \ge 2$, and Darmon and Merel \cite{MR1468926} for $r=1$.  
\end{proof}

\begin{lem}
Suppose that $r$ is a non-negative integer and $U,V,W$ are non-zero pairwise coprime integers. If 
\begin{equation} \label{lem2}
U^4-2^rV^4+W^4=0
\end{equation}
then $r=1$ and $|U|=|V|=|W|=1$. There are no solutions to the equation
\begin{equation} \label{lem3}
2^rU^4-V^4+W^4=0.
\end{equation}
 \end{lem}

\begin{proof}
See, for example, Section~6.5 
of \cite{MR2312337}. 
\end{proof}

\begin{proof}[Proof of Theorem \ref{first}] The only torsion points in $E_N(\mathbb{Q})$ are $2$-torsion. If $b$ is 
even then the rank of $E_N(\mathbb{Q})$ is zero (since it is zero when $b=0$), so assume that $b$ is odd.  Assume that $P \in 2E_N(\mathbb{Q})$ is non-zero.
The fundamental 2-descent map (see, for example, Section 8.2.3 in \cite{MR2312337}) shows that:
\begin{eqnarray*}
A_P&=&z_1^2; \\
A_P-2^ap^bB_P^2&=&z_2^2; \\
A_P+2^ap^bB_P^2&=&z_3^2. 
\end{eqnarray*}
Suppose that $p$ divides $A_P$ exactly $e$ times. Then $e$ is even, $e<b$, and, by replacing $A_P$ by $A_P/p^e$ and $b$ by 
$b-e$, we can assume that $p$ does not divide $A_P$. Equations (\ref{e3})-(\ref{e2}) become:
\begin{eqnarray*}
-2^ap^bB_P^2&=&(z_2-z_1)(z_2+z_1); \\
2^ap^bB_P^2&=&(z_3-z_1)(z_3+z_1); \\
2^{a+1}p^bB_P^2&=&(z_3-z_2)(z_3+z_2).
\end{eqnarray*}
Now $\gcd(z_j-z_i,z_j+z_i)$ divides $2z_j$ and $2^{a+1}p^bB_P^2$, so is a power of $2$. 

Suppose that $B_P$ is a perfect power. Now $p$ divides $z_2+(-1)^{s_1}z_1$ and $z_3+(-1)^{s_2}z_1$, where 
$s_1,s_2 \in \{ 0,1 \}$. So $p$ divides $z_3+(-1)^{s_3}z_2$, where $s_3=s_1+s_2+1$.
\it{Siegel's identity} \rm:
\[
(-1)^{s_3+1}\frac{z_2+(-1)^{s_1}z_1}{z_3+(-1)^{s_3}z_2}-\frac{z_3+(-1)^{s_2}z_1}{z_3+(-1)^{s_3}z_2}+1=0
\]
gives (\ref{lem}), (\ref{lem2}) or (\ref{lem3}).  Thus
\[
(-1)^{s_3+1}\frac{z_2+(-1)^{s_1}z_1}{z_3+(-1)^{s_3}z_2}=u \textrm{ and }  -\frac{z_3+(-1)^{s_2}z_1}{z_3+(-1)^{s_3}z_2}=v,
\]
where $(u,v)=(1,-2),(-2,1)$ or $(-\frac{1}{2},-\frac{1}{2})$. So
\[
-2^ap^bB_P^2=(-1)^{s_3+1}u(z_2+(-1)^{s_1+1}z_1)(z_3+(-1)^{s_3}z_2)
\]
and
\[
2^ap^bB_P^2=-v(z_3+(-1)^{s_2+1}z_1)(z_3+(-1)^{s_3}z_2).
\]
Dividing the two equations gives
\[
\frac{u}{v}=(-1)^{s_3+1}\frac{z_3+(-1)^{s_2+1}z_1}{z_2+(-1)^{s_1+1}z_1}.
\]
But
\[
\frac{z_3+(-1)^{s_3}z_2}{z_2+(-1)^{s_1+1}z_1}-\frac{z_3+(-1)^{s_2+1}z_1}{z_2+(-1)^{s_1+1}z_1}=(-1)^{s_3},
\]
thus $u \ne v$,
\[
\frac{z_2+(-1)^{s_1+1}z_1}{z_3+(-1)^{s_3}z_2}=(-1)^{s_3}\frac{v}{v-u},
\]
and
\[
\left( \frac{z_2+(-1)^{s_1+1}z_1}{z_3+(-1)^{s_3}z_2} \right) \left( \frac{z_2+(-1)^{s_1}z_1}{z_3+(-1)^{s_3}z_2} 
\right)=\frac{uv}{u-v}=
\frac{-2^ap^bB_P^2}{(z_3+(-1)^{s_3}z_2)^2}.
\]
So
\[
\frac{-uv}{2^ap^b(u-v)}
\]
is a square. Hence $(u,v)=(1,-2)$, $p=3$, $a$ is odd, $B_P=1$ and 
\[
(z_3+(-1)^{s_3}z_2)^2=2^{a-1}3^{b+1}.
\]
Thus
\[
z_3=\frac{1}{2}(z_3+(-1)^{s_3}z_2+z_3+(-1)^{s_3+1}z_2)=
\pm \frac{1}{2} \left( 2^{\frac{a-1}{2}}3^{\frac{b+1}{2}}+\frac{2^{a+1}3^b}{2^{\frac{a-1}{2}}3^{\frac{b+1}{2}}}  \right)
\]
and
\[
A_P=z_3^2-2^a3^b=2^{a-3}3^{b-1}25.
\]
From which it follows that $a \ge 3$ and $P$ is as required. \end{proof}

\begin{proof}[Proof of Theorem \ref{second}] Let $N=2^ap$ where $a=0$ or $1$. Let $P \in E_N(\mathbb{Q})$ non-torsion point with $x(P) \in -{\mathbb{Q}^*}^2$ and $x(P)+N \in p{\mathbb{Q}^*}^2$. If $m$ is even then the result follows from Theorem \ref{first}. If $m$ is odd then the fundamental 2-descent map (see, for example, 8.2.3 in \cite{MR2312337})  shows that $\alpha_1=-1$ and $\alpha_2=p$ so $\alpha_3=-p$.

Now (\ref{e4}) becomes $z_2^2+z_3^2=2^{a+1}B_P^2$ and (\ref{e2}) becomes $2z_1^2+pz_2^2=pz_3^2$ so 
\[
z_2^2+2p(z_1/p)^2=z_3^2.
\]
Corollary 6.3.6 in \cite{MR2312337} with the particular solution $(1, 0, 1)$ gives $dz_2=s^2-2pt^2$, $dz_1=2pst$, $dz_3=s^2+2pt^2$  where $s,t$ are coprime integers and $d \mid 2p$.

If $d=\pm 1$ then  $|z_2|=s^2-2pt^2$, $|z_1|=2pst$ and $|z_3|=s^2+2pt^2$. Since $z_1$ is even, $a=0$ and substituting into (\ref{e4}) gives $(s^2-2pt^2)^2+(s^2+2pt^2)^2=2B_P^2$ so 
\[
s^4+4p^2t^4=B_P^2.
\]
Now applying Theorem 1 in \cite{MR2646760} shows that $B_P$ cannot be a perfect power.

If $d=\pm 2$ then $|z_2|=2s^2-pt^2$, $|z_1|=2pst$ and $|z_3|=2s^2+pt^2$,  where $s,t$ are coprime integers. So $a=0$ and 
substituting into (\ref{e4}) gives 
\[
4s^4+p^2t^4=B_P^2.
\]
So $4s^4=B_P^2-p^2t^4=(B_P+pt^2)(B_P-pt^2)$. Since $2 \nmid B_P$ and $p \nmid B_P$, we have $B_P+pt^2=\pm 2s'^4$ and $B_P-pt^2=\pm 2t'^4$ where $s',t'$ are coprime and odd. Thus $\pm s'^4 \pm t'^4=B_P$. Again applying Theorem 1 in \cite{MR2646760} shows that if $B_P$ is a perfect power then it is a square or a cube, but these remaining cases are well known (see 6.5.2 of \cite{MR2312337} and 14.6.6 of \cite{MR2312338}).

Finally, the cases $d=\pm p$ and $d=\pm 2p$ give the same two parametrizations already considered above.       
\end{proof}

\bibliographystyle{amsplain}
\bibliography{myrefs}
\end{document}